\documentclass[11pt]{amsart}
\usepackage{amsmath,amssymb,amsthm,hyperref}
\usepackage{graphicx}
\usepackage{tikz}  
\usepackage{verbatim}
 \newtheorem{theorem}{Theorem}[section]
 \newtheorem{corollary}[theorem]{Corollary}
 
 \newtheorem{lemma}[theorem]{Lemma}
 \newtheorem{proposition}[theorem]{Proposition}
 
 {\theoremstyle{definition}
 \newtheorem{definition}[theorem]{Definition}
 \newtheorem{remark}[theorem]{Remark}

 }

\newcommand{\black}{\color{black}}

\def\R{\mathbb{R}}

\def\E{\mathbb{E}}
\def\Var{\mathbb{V}\mathrm{ar}}

\renewcommand{\qed}{\hfill$\square$}

\title[Total variation bound for Hadwiger-Wills information content]{\black \small Total variation bound for Hadwiger's functional using Stein's method}

\author{Valentin Garino}
\address{Valentin Garino, Uppsala University, 
Box 480, 751 06 Uppsala, Sweden}
\email{valentin.garino@math.uu.se} 
\thanks{}
\author{Ivan Nourdin}
\address{Ivan Nourdin, Universit\'e du Luxembourg, 
Unit\'e de Recherche en Math\'ematiques,
Maison du Nombre,
6 avenue de la Fonte,
L-4364 Esch-sur-Alzette,
Grand Duch\'e du Luxembourg}
\email{ivan.nourdin@uni.lu} 
\thanks{}

\date{\today}
\begin{document}
\maketitle

\medskip

\begin{abstract}
Let $K$ be a convex body in $\mathbb{R}^d$. Let $X_K$ be a $d$-dimensional random vector distributed according to the Hadwiger-Wills density $\mu_K$ associated with $K$, defined as $\mu_K(x)=ce^{-\pi {\rm dist}^2(x,K)}$, $x\in \mathbb{R}^d$. Finally, let the information content $H_K$ be defined as $H_K={\rm dist}^2(X_K,K)$.
The goal of this paper is to study the fluctuations of $H_K$ around its expectation as the dimension $d$ go to infinity.
Relying on Stein's method and Brascamp-Lieb inequality, we compute an explicit bound for the total variation distance between $H_K$ and its Gaussian counterpart.\\

\noindent{{\bf Keywords: } convex body, Steiner formula, intrinsic volumes, Wills functional, Stein's method, central limit theorem.\\
}
\end{abstract}

\section{Introduction}\label{Intro}

This paper is concerned with the asymptotic behavior of the intrinsic volumes of a convex body when 
the ambient dimension goes to infinity. Since our work is aimed at a rather probabilistic audience, let us start with a few reminders about these geometric
objects, before going into more detail on their probabilistic aspects.

\subsection{Intrinsic volumes of a convex body}\label{intrrr}

Throughout all the paper, $K\subset\R^d$ denotes a convex body, that is, a compact convex set with nonempty interior.
Its dimension, noted ${\rm dim} K$, takes values in $\{0,1,2,\ldots ,d\}$ and is by definition the dimension of its affine hull. 
When $K$ has dimension $j$, we define the $j$-dimensional volume ${\rm Vol}_j(K)$ to be the Lebesgue measure of $K$, computed relative to its affine hull.  
Finally, we write
$\mathbb{B}^j$ to indicate the Euclidean unit ball of $\R^j$.

The {\it Steiner formula} (e.g. \cite[Section 1]{Schneider2007Integral}) asserts that
${\rm Vol}_d(K+r\mathbb{B}^d)$ is a polynomial in $r>0$ written
\begin{equation}\label{steiner}
{\rm Vol}_d(K+r\mathbb{B}^d)=\sum_{k=0}^d \kappa_{d-k}\,v_k(K)\,r^{d-k},
\end{equation}
where the multiplicative constants $\kappa_{d-k}={\rm Vol}_{d-k}(\mathbb{B}^{d-k})=\pi^{\frac{d-k}2}/\Gamma(1+\frac{d-k}2)$ 
are here to guarantee that the $k$th intrinsic volume $v_k(K)$ is really intrinsic to $K$,
in the sense that it does not depend on the dimension in which the convex body is embedded (see below for a precise statement).

The commonly encountered intrinsic volumes are the {\it Euler characteristics} $v_0(K)=1$, the {\it intrinsic width} $v_1(K)$, (half) the {\it surface area} $v_{d-1}(K)$ and the {\it $d$-dimensional volume} $v_d(K)$. Simple properties satisfied by intrinsic volumes include:
\begin{itemize}
\item (positivity) $v_k(K)\geq 0$ for all $k\in\{0,\ldots,d\}$;
\item (homogeneity) $v_k(rK)=r^kv_k(K)$ for all $k\in\{0,\ldots,d\}$ and all $r\in(0,\infty)$;
\item (monotonicity) $v_k(K)\leq v_k(L)$ for all $k\in\{0,\ldots,d\}$, whenever $K\subset L$ are two convex bodies of $\R^d$;
\item (intrinsic) $v_k(K\times\{0\})=v_k(K)$ for all $k\in\{0,\ldots,d\}$,  whenever $K\subset\R^d$, $0\in\R^m$ and $K\times\{0\}\subset \R^{d+m}$;
\item (valuation) $v_k(K\cup L)+v_k(K\cap L)=v_k(K)+v_k(L)$ for all $k\in\{0,\ldots,d\}$, whenever $K,L\subset\R^d$ are two convex bodies such that $K\cup L$ is also a convex body;
\item (continuity) $v_k(K_n)\to v_k(K)$ if $K_n\to K$ in the Hausdorff metric.
\end{itemize}

There are only few examples of convex bodies for which we have explicit expression for their intrinsic volumes.
Two of them are the hypercube: 
\begin{equation}\label{binom}
v_k([0,1]^d)=\binom{k}{d},\quad 0\leq k\leq d,
\end{equation}
(see, e.g., Lotz et al.  \cite[Section 5]{lotz2020concentration} for more results about rectangular parallelotopes),
and the unit ball:
$$
v_k(\mathbb{B}^d)=\binom{k}{d}\frac{\kappa_d}{\kappa_{d-k}}, \quad 0\leq k\leq d.
$$
For further use, we observe
that 
\begin{equation}\label{further}
v_1(\mathbb{B}^d)\sim \sqrt{2\pi d}\,\mbox{ as $d\to\infty$}.
\end{equation}

In what follows, we will also rely on a deeper result,
observed independently by Chevet \cite{chevet1976processus} and Mc Mullen \cite{mcmullen1975non}  as a consequence of the Alexandrov-Fenchel inequalities:
\begin{equation}\label{ulc}
\mbox{The sequence $\{v_k(K)\}_{0\leq k\leq d}$ is ultra log-concave.}
\end{equation}
We recall that a sequence $(x_k)\subset\R_+$ is said to be {\it ultra log-concave} whenever it satisfies that $kx_k^2\geq (k+1)x_{k+1}x_{k-1}$ for all $k$.

\subsection{Hadwiger-Wills information content}\label{rv} 
Now let's bring probability into play.
Consider the {\it total} intrinsic volume of $K$, defined as
\begin{equation}\label{Wills}
W(K) = \sum_{k=0}^d v_k(K),
\end{equation}
and introduce a discrete random variable $V_K:\Omega\to\{0,\ldots,d\}$ distributed as
$$
\mathbb{P}(V_K=k)=\frac{v_k(K)}{W(K)},\quad 0\leq k\leq d.
$$

In the case $K=[0,1]^d$, we immediately deduce from (\ref{binom}) that $V_K$ is distributed according to the binomial distribution $\mathcal{B}(d,\frac12)$ and, thus, satisfies a {\it central limit theorem} when $d\to\infty$.
Is this a general phenomenon? 

Given the discrete nature of the problem and the fact that intrinsic volumes are difficult to estimate in high dimension, answering this question directly at a sufficient level of generality seemed too difficult to us at first attempt. This is why we decided to study the asymptotic behavior of $V_K$ in a somewhat diverted way, which we now describe.

In \cite{Had75} (see also \cite{W}), the following integral representation for the total intrinsic volume $W(K)$ is stated
$$W(K)=\int_{\R^d} e^{-\pi\,{\rm dist}^2(x,K)} dx,$$
where ${\rm dist}(\cdot,K)$ stands for the distance to $K$, defined for $x\in\R^d$ as $${\rm dist}(x,K)=\|x-\Pi_K(x)\|,$$ 
with $\Pi_K:\R^d\to K$ the projection onto $K$.
Introducing a parameter $\lambda>0$ and using the homogeneity of the $v_k(K)$'s leads to
\begin{eqnarray*}
&& \sum_{k=0}^d \lambda^k v_k(K) =W(\lambda K)= \int_{\R^d} e^{-\pi\,{\rm dist}^2(x,\lambda K)} dx= \lambda^d \int_{\R^d} e^{-\pi\,{\rm dist}^2(\lambda x,\lambda K)} dx,
\end{eqnarray*}
implying in turn
$$
\sum_{k=0}^d \lambda^{k-d} v_k(K) =  \int_{\R^d} e^{-\pi\,\lambda^2 {\rm dist}^2(x,K)} dx.
$$
Now, observe that
$$
\lambda^{k-d} = \frac{1}{\Gamma(\frac12(d-k))}\int_0^\infty r^{-1+\frac{d-k}2}e^{-\lambda^2 r}dr, \quad k=0,\ldots,d-1.
$$
We deduce, with $g(r)=e^{(1-\lambda^2)r}$
\begin{eqnarray}
&&\int_{\R^d} g(\pi\,{\rm dist}^2(x,K))e^{-\pi\,{\rm dist}^2(x,K)}dx \label{distanceintegral}\\
&=& g(0)v_d(K) + \sum_{k=0}^{d-1} \left(\frac{1}{\Gamma(\frac12(d-k))}\int_0^\infty g(r)r^{-1+\frac{d-k}2}e^{-r}dr\right) v_j(K).\notag
\end{eqnarray}
By linearity and density, (\ref{distanceintegral}) is actually valid for any integrable and continuous function $g:\R_+\to\R$
(see also \cite[Corollary 2.5]{lotz2020concentration} for a different proof).

Now, let us introduce $X_K:\Omega\to\R^d$ with the log-concave density
\begin{equation}\label{muK}
x\in\R^d\mapsto \frac{1}{W(K)}e^{-\pi\,{\rm dist}^2(x,K)},
\end{equation}
that we name {\it Hadwiger-Wills density associated to} $K$, in honor of the influential papers \cite{Had75} and \cite{W}.
Finally, let us define $H_K:\Omega\to\R_+$ as
$$H_K(\omega):=\pi\,{\rm dist}^2(X_K(\omega),K).$$ 
Identity (\ref{distanceintegral}) has the following probabilistic interpretation:
\begin{equation}\label{idinlaw}
H_K\overset{\rm law}{=} \sum_{j=1}^{d-V_K}\gamma_j,
\end{equation}
where the $\gamma_j$ are independent $\Gamma(\frac12,1)$ random variables, supposed to be also independent of $V_K$.

Given the identity in law (\ref{idinlaw}) (from which it comes that any probabilistic statement for $H_K$ transfers to $V_K$, and vice versa), we chose to focus on obtaining a limit theorem for $H_K$ instead of
$V_K$, the first seeming more amenable to analysis.

As a last remark, we note that $H_K$ is nothing but the (shifted) {\it information content} (or {\it Shannon entropy}) of $X_K$,
a property that  we will not use here, but which was crucial in Lotz {\it et al} \cite{lotz2020concentration} to prove that $H_K$ (and then $V_K$) displays a form of concentration. 

\subsection{Our main result}
If $F,G:\Omega\to\R$ are two random variables, recall that the total variation distance between $F$ and $G$ is defined as 
$$d_{TV}(F,G)=\sup_{A\in\mathcal{B}(\mathbb{R})}|\mathbb{P}(G\in A)-\mathbb{P}(F\in A)|.$$
The goal of this paper is to prove the following statement.
\begin{theorem}\label{main}
When $K$ denotes a convex body of $\R^d$, we let $X_K:\Omega\to\R^d$ have the density (\ref{muK}),
and we define $F_K:\Omega\to\R$ as
$$F_K=\frac{
{\rm dist}^2(X_K,K)
-\mathbb{E}[{\rm dist}^2(X_K,K)]
}
{\sqrt{
\Var\big(
{\rm dist}^2(X_K,K)
\big)
}}.$$
Also, we let $N\sim\mathcal{N}(0,1)$ to be a standard Gaussian random variable.

For all $\alpha\geq 0$ and $c>0$, we have
\begin{equation}\label{main1}
\sup_{d\geq 1} \,\,\sup_{K:\,{\rm diam}(K)\leq 2c\,d^\alpha} d^{\frac12-\alpha}d_{TV}(F_K,N)<\infty,
\end{equation}
where the second supremum runs over all convex bodies $K$ of $\R^d$ whose diameter is bounded by $2c\,d^\alpha$.

In particular, the following central limit theorem takes place: if $\alpha\in[0,\frac12)$,
a sequence $(d_n)$ of natural numbers diverging to infinity and a sequence $(K_n)$ of convex bodies
satisfying $K_n\subset \R^{d_n}$ and ${\rm diam}(K_n)=O(d_n^\alpha)$ are given,  then
\begin{equation}\label{(5)}
F_{K_n}\overset{\rm law}{\to} \mathcal{N}(0,1)\quad\mbox{as $n\to\infty$}.
\end{equation}
\end{theorem}

We note that Theorem \ref{main} parallels recent fluctuation results proved in the context of {\it conic} intrinsic volumes, see
\cite{goldstein2017gaussian} and more precisely Theorems 1.1, 2.1 and 3.1 therein.
However, and unlike \cite{goldstein2017gaussian}, it is unfortunately not possible to deduce from (\ref{(5)}) that
a CLT also holds for $(V_{K_n})$, which would have positively answered our original question, see the beginning of Section \ref{rv}. Investigating whether and when such a CLT takes place
or not will be the object of another work.

\subsection{Sketch of the proof of (\ref{(5)})}
Of course, (\ref{(5)}) is as an immediate consequence of (\ref{main1}). 
However, we would like to describe here another, more direct, approach, which we find interesting
for itself.

We can deduce (see Section \ref{sec31} for the details) from (\ref{idinlaw}) and the usual central limit theorem that 
$$
\frac{{\rm dist}^2(X_K,K)-\E[{\rm dist}^2(X_K,K)]}{\sqrt{\Var({\rm dist}^2(X_K,K))}}\approx \mathcal{N}(0,1)
$$as soon as
\begin{equation}\label{(9)}
\frac{\Var(V_K)}{d-\mathbb{E}[V_K]}=o(1).
\end{equation}
This is in order to check  (\ref{(9)}) that we need to assume  that the diameter of $K$ is $O(d^\alpha)$ with $\alpha<\frac12$.
Indeed this assumption allows us to write 
\begin{equation}\label{a}
v_1(K) \leq v_1(O(d^{\alpha})\mathbb{B}^d) = O(d^\alpha) \,v_1(\mathbb{B}^d) = O(d^{\alpha+\frac12}),
\end{equation}
where the last equality comes from (\ref{further}).
Now, exploiting the ultra log-concavity  (\ref{ulc}) of $(v_k(K))$, we can also prove 
\begin{equation}\label{b}
\mathbb{E}[V_K]\leq v_1(K)\quad\mbox{and}\quad
\Var(V_K)\leq v_1(K),
\end{equation} 
see Corollary \ref{BoundIntr}. Putting (\ref{a}) and (\ref{b}) together, the desired conclusion (\ref{(9)}) follows now immediatly,
implying in turn (\ref{(5)}).

\subsection{Sketch of the proof of (\ref{main1})}
First, we apply Stein's method to prove the following (general) statement on Gibbs distribution.

Let $X:\Omega\to\R^d$ have a density of the form $e^{-\phi}$ with $\phi:\R^d\to\R$ regular enough.
Consider an independent copy $\widehat{X}$ of $X$, and set
\begin{eqnarray*}
Y&=&\nabla \phi(X),\quad \widehat{Y}=\nabla \phi(\widehat{X})\\
Y_t&=&e^{-t}Y+\sqrt{1-e^{-2t}}\widehat{Y},\quad t\in[0,\infty].
\end{eqnarray*}
Finally, let $H:\R^d\to\R$ be smooth, and set
$$
F=\frac{H(Y)-\mathbb{E}[H(Y)]}{\sqrt{\Var(H(Y))}}.
$$
Write `${\rm Hess}$' to denote the Hessian operator.
Then $d_{TV}(F,\mathcal{N}(0,1))\leq A+B$,
where
\begin{eqnarray*}
A&=&\frac{\rm 2}{\sigma^2} \sqrt{\Var\left(
\int_0^\infty e^{-t}\big\langle 
({\rm Hess} \phi)(X) \nabla H(Y), \widehat{\mathbb{E}}[\nabla H(Y_t)]
\big\rangle dt
\right)}\\
B&=&\sup_{g}\big|\mathbf{E}[C(g)]\big|,
\end{eqnarray*}
with $$C(g)=\frac{1}{\sigma}\int_0^{\infty}e^{-2t}(g(F)-F){\rm Tr}\left[\big(({\rm Hess}\phi)(X)-({\rm Hess}\phi)(X_{\infty})\big)({\rm Hess}H)(Y_t)\right]dt$$
and where the supremum runs over the set of all functions $g\in\mathcal C^1$ such that $|g(x)|\leq \sqrt{\frac{\pi}{2}}$ and $|g'(x)|\leq 2$ for all $x\in\R$.

In our case, we have $$\phi(x)=\pi\,{\rm dist}^2(x,K)+\log(W(K))$$ (implying $\nabla \phi(x)=2\pi(x-\Pi_K(x))$) and
$H(x)=\|x\|^2$, which gives that
$d_{TV}\left(F_K,\mathcal{N}(0,1)\right)$ is less than a universal constant times
$$
\frac{1}{\sigma^2} \sqrt{\Var\left(
\int_0^\infty e^{-t}\big\langle 
Y, \widehat{\mathbb{E}}[\nabla H(Y_t)]
\big\rangle dt
\right)}+\frac{1}{\sigma} 
\sqrt{\Var
\Big({\rm Tr}\big(
({\rm Hess} \phi)(X)
\big)\Big)
}.
$$
To get (\ref{main1}), it remains to bound these two variances.
For this, we mainly rely on the Brascamp-Lieb inequality.

\subsection{Organisation of the paper}
The rest of the paper is organised as follows.  
Section \ref{Prlms}  contains a few preliminaries needed for the proof of Theorem \ref{main}. In Section \ref{further} we make additional remarks outlining potential improvements and continuations of this work. Section \ref{Stein} contains our results regarding Stein's method for Gibbs distributions, which might be of independent interest.
The proof of Theorem \ref{main} is finally done in Section \ref{mainProof}.

\section{Preliminaries}\label{Prlms}
We gather in this section a few preliminaries regarding continuous log-concave and discrete ultra log-concave probability measures, as well as intrinsic volumes of a convex body. These facts will be used in the proof of 
Theorem \ref{main}.

\subsection{Continuous log-concave probability measures}
\begin{definition}
Let $\mu$ be a continuous probability law on $\R^d$ with a density of the form $p=e^{-\theta}$ with $\theta:\R^d\to\R$.
\begin{enumerate}
\item If $\theta$ is convex, we say that $\mu$ is \textit{log-concave}.
\item Fix $k>0$. If $\theta$ is $C^2$ and satisfies
$\langle ({\rm Hess}\theta)(x)u,u\rangle\geq k\|u\|^2$ for all $x,u\in\R^d$, we say that $\mu$ is 
\textit{$k$-strongly log-concave}.
\end{enumerate}
\end{definition}

Strongly log-concave probability measure satisfy the celebrated Brascamp-Lieb inequality
\cite{Brascamp2002extensions}.
\begin{proposition}\label{BrascampLieb}
Fix $k>0$ and let $\mu$ be a continuous $k$-strongly log-concave probability measure 
on $\R^d$ with density $p$.
Assume that $X:\Omega\to\R$ is distributed according to $\mu$.
Then, for any differentiable function $f:\mathbb{R}^d\to \mathbb{R}$ such that $\Var(f(X))<\infty$,
\begin{eqnarray}
\Var(f(X))&\leq& \mathbb{E}\left[(\nabla f(X))^T({\rm Hess}(-\log p)(X))^{-1}\nabla f(X)\right].
\label{BC}
\end{eqnarray}
\end{proposition}

\subsection{Discrete ultra log-concave probability measures}

\begin{definition}
Let $d\in\mathbb{N}$ and let $(x_k)_{0\leq k\leq d}$ be a
sequence of nonnegative numbers.
We say that $(x_k)$ is log-concave if 
\begin{equation}\label{LC}
x_k^2\geq x_{k-1}x_{k+1}\quad\mbox{for all $k\in \{1,\ldots,d-1\}$}.
\end{equation}
We say that $(x_k)$ is ultra log-concave if 
\begin{equation}\label{ULC}
kx_k^2\geq (k+1)x_{k-1}x_{k+1} \quad\mbox{for all $k\in \{1,\ldots,d-1\}$}.
\end{equation}
\end{definition}

It is immediate to check that $(x_k)$ is ultra log-concave if and only if $(k!x_k)$ is log-concave.
Discrete ultra log-concave probability measures satisfy a Poincar\'e-type variance bound.

\begin{proposition}\label{ULCVar}
Let $(x_k)_{0\leq k\leq d}$ be an ultra log-concave distribution, that is,
an ultra log-concave sequence of nonnegative numbers summing up to 1.
Let $V$ be a random variable on $\{0,\ldots,d\}$ distributed according to $x$, that is, $\mathbb{P}(X=k)=x_k$, $0\leq k\leq d$.
Then
\begin{enumerate}
\item[(a)] $\mathbb{E}[V]\leq \frac{x_1}{x_0}$;
\item[(b)] $\Var(V)\leq \frac{x_1}{x_0}$.
\end{enumerate}
\end{proposition}
\begin{proof}
It is immediate to check that ultra log-concavity of $(x_k)$ is equivalent to the fact that the sequence
$\left(
\frac1{k+1}\,\frac{x_k}{x_{k+1}}
\right)$
is increasing. We deduce that
$$
\frac{x_k}{x_{k+1}} - \frac{x_{k-1}}{x_k}\geq \left(\frac{k+1}{k}-1\right)\frac{x_{k-1}}{x_k}
=\frac1k\,\frac{x_{k-1}}{x_k}\geq \frac{x_0}{x_1},
$$
where in the last inequality we used again that $\left(
\frac1{k+1}\,\frac{x_k}{x_{k+1}}
\right)$
is increasing.

This shows that $(x_k)$ is $c$-log-concave with $c=\frac{x_0}{x_1}$ in the sense of \cite{johnson2017discrete}.
We can then apply \cite[Lemma 5.3]{johnson2017discrete} to obtain (a) and
\cite[Theorem 1.5]{johnson2017discrete} to obtain (b). 
\end{proof}

\subsection{Elements of convex geometry}
In this section, $K\subset\mathbb{R}^d$ denotes a convex body, that is, a non empty, compact and convex set. It is well-known that for every point $x\in\mathbb{R}^d$ there is a unique point in $K$, denoted by $\Pi_K(x)$ and called the projection of $x$ onto $K$, such that ${\rm dist}(x,K)=\|x-\Pi_K(x)\|$. Moreover, the application $x\mapsto {\rm dist}(x,K)$ is continuous. We also have the following fact, proven e.g. in \cite{goldstein2017gaussian}.

\begin{lemma}
The gradient of the square distance to $K$ is given, for almost all $x\in\mathbb{R}^d$, by
$$\nabla {\rm dist}^2(x,K)=2(x-\Pi_K(x)).$$
We also have that
\begin{equation}\label{relation}
\nabla \Pi_K(x) (x-\Pi_K(x))=0
\end{equation}
for all $x\in\R^d$.
\end{lemma}

We now turn our attention to the intrinsic volumes, as defined in Section \ref{intrrr}. 
Let us state further facts relevant to polytopes, 
which will be useful in our proof of Theorem \ref{main}.
\begin{proposition}\label{polytopesProp}
Let $K\subset\mathbb{R}^d$ be a polytope. For all $r>0$, let $S_r$ denote the boundary of the outer parallel body $K+r\mathbb{B}^d$, endowed with the $(d-1)$-dimensional uniform measure. Then, the following facts hold:
\begin{enumerate}
\item Let $X_K:\Omega\to\R^d$ be a random vector distributed according to the Hadwiger-Wills density (\ref{muK}) associated with $K$. The law of $X_K$ conditional to ${\rm dist}(X_K,K)$ is the uniform law on the $(d-1)$-dimensional surface
$S_{{\rm dist}(X_K,K)}$.
\item If $x\in \overset{\circ}{K}$, ${\rm Tr}(\nabla\Pi_K(x))=d$.
\item Let $F$ be a $i$-dimensional face of $K$. Let $C_F\subset \mathbb{R}^d$ be the set of points $x\in\mathbb{R}^d$ such that $\Pi_K(x)\in \overset{\circ}{F}$, where $\overset{\circ}{F}$ is the relative interior of $F$. 
Then ${\rm Tr}(\nabla\Pi_K(x))=i$ for all $x\in C_F$.
\item For each face $F$, the boundary of $C_F$ has null Lebesgue measure. 
\item For $r>0$ and $\mathcal F_i$ the set of $i$-dimensional faces of K, the surface area of $\cup_{F\in\mathcal F_i}C_F\cap S_r$ is given by $${\rm Vol}_{d-1}(\cup_{F\in\mathcal F_i}C_F\cap S_r)=r^{d-i-1}v_{i}(K)(d-i)\kappa_{d-i}.$$
\end{enumerate}
\end{proposition}
\begin{proof}
{\it Proof of (1)}. It simply follows from the coarea formula 
$$
\int_{\R^d\setminus K} F(x)dx = \int_{\R_+}\left(
\int_{({\rm dist}(\cdot,K))^{-1}(r)}F(x)\frac{d\mathcal{H}^{d-1}(x)}{\big|\nabla {\rm dist}(x,K)\big|}
\right)dr
$$
and the fact that $\big|\nabla {\rm dist}(x,K)\big|=1$.

{\it Proof of $(2)$}. It follows from the fact that $\Pi_K$ coincide with the identity function in the interior of $K$. 

{\it Proof of $(3)$}. 
Let $(e_1,\ldots,e_i)$ be an orthonormal basis of the linear hull of 
an $i$-dimensional face $F$  and let $(e_{i+1},\ldots,e_d)$ be a basis of the orthogonal vector space to $F$. Then, $\nabla\Pi_K(x)$ is given by the linear application $e_k\mapsto e_k$ for $k\leq i$ and $e_k\mapsto 0$ for $k>i$. This proves $(3)$. 

{\it Proof of $(4)$}. 
It is well known that $\Pi_K$ is a Lipschitz application, hence differentiable almost everywhere. Since $\Pi_K$ is not differentiable on $\partial(C_F)$, it means that $\partial(C_F)$ has null measure, which proves $(4)$.

{\it Proof of $(5)$} 
We follow the proof of the Steiner formula from \cite[pp. 5-6]{Schneider2007Integral}. Let $P\subset\mathbb{R}^d$ be a polytope. Let $r>0$ and let $S_{r,P}$ be the surface of the parallel body $P+r\mathbb{B}^d$. For $k\in \{1,\ldots,d\}$, let $\lambda_k$ denotes the $k$-dimensional Lebesgue measure. Let $\mathcal F_m$ be the set of $m$-dimensional faces of $P$. If $F$ is a face of $P$, let $N(F,P)$ be the normal cone of $P$ at $F$, that is the closed convex cone of outward normal vectors of $P$ at any point $x\in\mathring{F}$ (it does not depend on the choice of $x$). 
    Then, if $F$ is a $m$-dimensional face   ($m\in\{0,\ldots,d-1\}$), then $$\gamma(F,P):=\frac{\lambda_{d-m}(N(F,P)\cap (d\mathbb{B}^d))}{\kappa_{d-m}}$$ is called the external angle of $P$ at $F$. We also put $\gamma(P,P)=1$.
     
    \smallbreak We can now write 
    \begin{equation*}
    S_{r,K}=\bigcup_{m=0}^d\bigcup_{F\in \mathcal F_m } C_F\cap S_{r,P}.
    \end{equation*}
    Since this union is disjointed and since the boundary of $C_F$ has null measure (by item $(4)$), we can then write
    \begin{equation}\label{decomposition}
    {\rm Vol}_{d-1}(S_{r,P})=\sum_{i=0}^{m-1}\sum_{F\in\mathcal F_m}{\rm Vol}_{d-1}(C_F\cap S_{r,P})=\sum_{i=0}^m\sum_{F\in\mathcal F_m}{\rm Vol}_{d-1}(\mathring{C_F}\cap S_{r,P}).
    \end{equation}
   
    Let $F\in\mathcal F_m$. We have $C_F\cap S_{r,P}=\mathring{C_F}\oplus N(F,P)\cap r\mathbb{S}^{d-1}$, where $\mathbb{S}^{d-1}$ is the unit sphere. Thus, we have 
    \begin{eqnarray*}{\rm Vol}_{d-1}(C_F\cap S_{r,P})&=&\int_{F}\left(\int_{\mathbb{S}^{d-m-1}}\mathbf{1}_{x+ry\in N(F,P)\cap r\mathbb{S}^{d-1}}d\lambda_{d-m-1}(y)\right)d\lambda_{m}(x)\\
                                                                       &=&\lambda_m(F)\gamma(F,P)r^{d-m-1}(d-m)\kappa_{d-m}.
    \end{eqnarray*}
    Thus, 
       \begin{eqnarray*}
       \sum_{F\in\mathbb{F}_m}
       {\rm Vol}_{d-1}(C_F\cup S_{r,P})&=&\sum_{F\in\mathbb{F}_m}\lambda_m(F)\gamma(F,P)r^{d-m-1}(d-m)\kappa_{d-m}\\
       &=&r^{d-m-1}V_m(P)(d-m)\kappa_{d-m}.
       \end{eqnarray*}
    Finally, plugging this in \eqref{decomposition} concludes the proof of item $(5)$.

\end{proof}
\medskip 

We end this section by (re-)introducing the intrinsic volume random variable, and by briefly discussing its ultra log-concavity. As noted in the introduction, the \textit{Wills functional} $W(K)=\int_{\mathbb{R}^d}e^{-\pi {\rm dist}^2(x,K)}dx$ satisfies $W(K)=\sum_{k=0}^dv_k(K)$. Renormalising the sequence $v_j(K)$ by $W(K)$ then gives a probability distribution on $\{0,\ldots, d\}$.
\begin{definition}
Let $K\subset\R^d$ be a convex body.
The intrinsic volume random variable associated with $K$ is defined as the random variable $V_K:\Omega\to\{0,\ldots,d\}$ such that 
$$\mathbb{P}[V_K=k]=\frac{v_k(K)}{W(K)},\quad\mbox{for $k\in \{0,\ldots,d\}$.}$$
\end{definition}

It is well-known that the sequence of intrinsic volumes $(v_k(K))_{0\leq k\leq d}$ is ultra-log concave. This fact is proved in e.g \cite[Lemma 4.2]{chevet1976processus}. Proposition \ref{ULCVar} then gives the following corollary:
\begin{corollary}\label{BoundIntr}
Let $K\subset\R^d$ be a convex body and let $V_K$ be the intrinsic volume random variable associated with $K$.
We have $\mathbb{E}[V_K]\leq v_1(K)$ and $\Var(V_K)\leq v_1(K)$.
\end{corollary}
\begin{remark} 
It is shown in \cite{lotz2020concentration} that $\Var(V_K)\leq 4d$. The bound $\Var(V_K)\leq v_1(K)$ is actually better as long as $K\subset d^{\alpha}\mathbb{B}^d$ with $\alpha\in[0,\frac12)$, thanks to the monotonicity property
(see also (\ref{hello})). On the other hand, Lotz and Tropp \cite{LotzTropp} have recently shown the bound
$$\Var(V_K)\leq 2\,\frac{d-\mathbb{E}[V_K]}{d+\mathbb{E}[V_K]}\,\mathbb{E}[V_K].$$
Even if won't need their bound in our approach, we note that it is always sharper (up to a factor 2) than the bound
from Corollary \ref{BoundIntr}, thanks to Proposition \ref{ULCVar}(a).
\end{remark}

\section{Further comments}\label{further}
\subsection{Convergence in law in Theorem \ref{main}}\label{sec31}
Establishing the mere convergence in law instead of the convergence in total variation in Theorem \ref{main} is actually quite simple and follows from elementary properties of the intrinsic volumes listed above. 

Indeed, let $K$ be a convex body
of $\R^d$ and let $V_K:\Omega\to\{0,\ldots,d\}$ be the intrinsic random variable associated with it.
Let also $X_K:\Omega\to\R^d$ be a random vector distributed according to the Hadwiger-Wills density
(\ref{muK})  associated with $K$.

We deduce from (\ref{idinlaw}) that
\begin{eqnarray*}
&&\frac{{\rm dist}^2(X_K,K)-\mathbb{E}[{\rm dist}^2(X_K,K)]}{\sqrt{\Var[{\rm dist}^2(X_K,K)]}}\\
&\overset{\mathcal{L}}{= }&\frac{1}{\sigma_K}\sum_{k=1}^{d-V_K}\left(\gamma_k-\frac{1}{2}\right)
+\frac{1}{2\sigma_K}(d-V_K-2\mathbb{E}[\pi {\rm dist}^2(X_K,K)]),
\end{eqnarray*}
where $\sigma^2_K=\Var(\pi {\rm dist}^2(X_K,K))$. 
We then deduce the following relationships, whose
proofs are straightforward and left to the reader.
\begin{lemma}\label{relationsVarEsp}
Let $\delta_K=\mathbb{E}[{\rm dist}^2(X_K,K)]$, $\sigma^2_K=\Var(\pi {\rm dist}^2(X_K,K))$ and  $\tau_K^2=\Var(V_K)$. Then
\begin{enumerate}
\item $\delta_K=\frac{d-\mathbb{E}[V_K]}{2}$;
\item $\sigma_K^2=\frac{\tau_K^2}{4}+\delta_K$.
\end{enumerate}
\end{lemma}

Applying Lemma \ref{relationsVarEsp}, we can write
\begin{equation}\label{15}
\mathbb{E}\left[
\left\{
\frac1{2\sigma_K}(d-V_K-2\delta_K)
\right\}^2
\right] = \frac{\tau^2_K}{4\sigma^2_K}=\frac1{1+\frac{4\delta_K}{\tau^2_K}}
\end{equation}
and
\begin{equation}\label{16}
\mathbb{E}\left[
\left\{
\frac1{\sigma_K}\sum_{k=d-V_K}^d (\gamma_k-\frac12)
\right\}^2
\right] = \frac{\mathbb{E}[V_K]}{\frac12\tau^2_K+2\delta_K}
\leq \frac{\mathbb{E}[V_K]}{2\delta_K}.
\end{equation}
Thanks to Proposition \ref{ULCVar} and the fact that intrinsic volumes form an ultra log-concave sequence, we have that 
$\mathbb{E}[V_K]\leq v_1(K)$ and $\tau^2_K\leq v_1(K)$.

Assume now that $K\subset c\,d^\alpha \mathbb{B}^d$ with $c>0$, $\alpha\in[0,\frac12)$ and $d\to\infty$.
We deduce that 
\begin{equation}\label{hello}
v_1(K)\leq v_1(c\,d^\alpha\mathbb{B}^d)=c\,d^{1+\alpha}\sqrt{\pi}\frac{\Gamma(\frac{d}2+\frac12)}{\Gamma(\frac{d}2+1)}
\sim c\,\sqrt{2\pi}\,d^{\frac12+\alpha} = o(d),
\end{equation}
where the equivalent comes from the fact that
$\Gamma(x+1)\sim \sqrt{2\pi}\,x^{x+\frac12}e^{-x}$ and $\Gamma(x+\frac12)\sim \sqrt{2\pi}\,x^x\,e^{-x}$ 
as $x\to\infty$. We deduce that $\delta\sim \frac{d}2$, that $\mathbb{E}[V_K]=o(d)$ and that
$\tau^2_K=o(d)$ as $d\to\infty$. As a consequence, the left-hand sides of (\ref{15}) and (\ref{16}) go to
zero as $d\to\infty$.

On the other hand, as $d\to\infty$ we have $\sigma^2_K\sim \frac{d}2$ and, by the classical central limit theorem,
 $$
 \frac{1}{\sigma_K}\sum_{k=1}^d\left(\gamma_k-\frac{1}{2}\right)
 =\frac{\sqrt{d}}{\sigma_K}\frac{1}{\sqrt{d}}\sum_{k=1}^d\left(\gamma_k-\frac{1}{2}\right)\overset{\mathcal L}{\longrightarrow} \mathcal{N}(0,1).
 $$
Finally, Slutsky's lemma allows to conclude that
$$
\frac{
{\rm dist}^2(X_K,K)
-\mathbb{E}[{\rm dist}^2(X_K,K)]}
{
\sqrt{\Var[{\rm dist}^2(X_K,K)]}
}
\overset{\mathcal L}{\longrightarrow} \mathcal{N}(0,1).
$$
\qed

\begin{remark}
Because we relied above on a `Slutsky-type' arguments, 
we cannot easily pass from this convergence in law to a convergence in total variation.
New arguments are thus needed.
\end{remark}
\subsection{Case of the cube}
The main assumption of Theorem \ref{main} 
(namely, $K\subset c\,d^\alpha \mathbb{B}^d$ with $c>0$ and $0\leq \alpha<\frac12$) 
is not sharp, as is shown by investigating the case of the scaled cubes.  
More precisely,
assume that $K$ is a hypercube of the form $[-T_d,T_d]^{d}$ with $T_d>0$.
It is immediate to check that 
\begin{equation}\label{formula}
{\rm dist}^2(x,K) = \sum_{k=1}^{d} (|x_k|-T_d)_+^2
\end{equation}
for all $x= (x_1,\ldots,x_{d})\in\R^{d}$, where $(\ldots)_+^2$ is shorthand for $[(\ldots)_+]^2$. 
This implies that the marginals of $X_{d}$ are i.i.d.,
with a common density given by  $u\rightarrow \frac{e^{-\pi\, (|u|-T_d)_+^2}}{1+2T_d}$. Furthermore, these marginals all have an absolutely continuous part with respect to the Lebesgue measure, and the variance verifies $\Var(H_{d}) = \frac{8\pi^2d(1+3T_d)}{(1+2T_d)^2}$. According e.g. to \cite{BAs16}, 
$$d_{TV}\left(\frac{H_{d}-\mathbb{E}[H_{d}]}{\sqrt{\Var(H_{d})}},\mathcal{N}(0,1)\right)=O\left(\sqrt{\frac{T_d\vee 1}{d}}\right).$$
That is, the CLT in total variation 
holds true if $T_d=o(d)$, whereas Theorem \ref{main} allows to prove it only for $K\subset{c\,d^{\alpha}\mathbb{B}^d}$ with $c>0$ and $\alpha<\frac{1}{2}$.

\section{Stein's method for Gibbs measures}\label{Stein}
As stated in the introduction, the main ingredient for our proof is the (one-dimensional) Stein's method, which makes an extensive use of the following integration by part formula, known as Stein's lemma, initially stated for Gaussian random variables and which we state below in a slightly extended setting.
\begin{lemma}\label{IPP}
Let $\mu$ be a probability law, whose density is given by \begin{equation}\label{Gibbs}g_\mu:=e^{-\phi}\quad\mbox{with $\phi\in\mathcal{C}^1(\mathbb{R}^d,\mathbb{R})$.}\end{equation}
Let $X\sim\mu$.
Then, for all absolutely continuous function $f:\mathbb{R}^d\rightarrow\mathbb{R}^d$ such that
\begin{equation}\label{Cond1}\mathbb{E}[|f(X)|]+\mathbb{E}[|\langle f(X),\nabla \phi(X)\rangle_{\mathbb{R}^d}|]+\mathbb{E}[|{\rm Tr}(\nabla f(X))|]<\infty,\end{equation}
we have
\begin{equation}\label{IPPformula}
\mathbb{E}[\langle f(X),\nabla \phi(X)\rangle_{\mathbb{R}^d}]=\mathbb{E}[{\rm Tr}(\nabla f(X))].
\end{equation}
\end{lemma}

\begin{proof}
We first consider the case where $f$ has compact support. 
We have, thanks to Fubini's theorem and the fact that $ |f_i(x)|e^{-\phi(x)}\to 0$ for all $i\in\{1,\ldots,n\}$ and $\|x\|\to\infty$,

\begin{eqnarray*}
\mathbb{E}[\langle f(X),\nabla\phi(X)\rangle_{\mathbb{R}^d}] &=&\int_{\mathbb{R}^d}\left(\sum_{i=1}^df_i(x)\frac{\partial \phi}{\partial x_i}(x)\right)e^{-\phi(x)}dx\\
&=&\sum_{i=1}^n\int_{\mathbb{R}^{d-1}}\left(\int_{\mathbb{R}}f_i(x)\frac{\partial \phi}{\partial x_i}(x)e^{-\phi(x)}dx_i\right)\prod_{j\neq i}dx_j\\
&=&\sum_{i=1}^n\int_{\mathbb{R}^{d-1}}\left(\int_{\mathbb{R}}\frac{\partial f_i}{\partial x_i}(x)e^{-\phi(x)}dx_i\right)\prod_{j\neq i}dx_j\\
&=&\mathbb{E}[{\rm Tr}(\nabla f(X))].\\
\end{eqnarray*}

Now, let $f$ be a function satisfying the hypotheses of Lemma \ref{IPP}, and let $(p_n)_{n\in\mathbb{N}^*}$ be a sequence of smooth functions from $\mathbb{R}^d$ to $\mathbb{R}^d$ such that $p_n(x)=1$ if $\|x\|\leq n$, $p_n(x)=0$ if $\|x\|\geq n+1$ and $\sup_{n}\|\nabla p_n\|_{\infty}<\infty$. Then, the sequence $u_n=(fp_n)_{n\in\mathbb{N}^*}$ verifies 
$$\mathbb{E}[\langle u_n(X),\nabla \phi(X)\rangle_{\mathbb{R}^d}]=\mathbb{E}[{\rm Tr}(\nabla u_n(X))].$$
Thanks to \eqref{Cond1}, we can apply the dominated convergence theorem from which the desired conclusion follows. \end{proof}
When the target law is the normal distribution and the total variation is the chosen distance, one of the main achievement of Stein's method is the following bound (see e.g. \cite{nourdin2012normal}).
\begin{proposition}
Let $F$ be a real valued random variable and let $N$ be distributed according to the standard normal law. Then,
\begin{equation}\label{boundTotVar}
d_{TV}(F,N)\leq \sup_{g}\big|\mathbb{E}[g'(F)-Fg(F)]\big|,
\end{equation}
where the supremum runs over the set of all absolutely continuous real valued functions $g$ such that $\|g\|_{\infty}\leq \frac{\sqrt{\pi}}{2}$ and $\|g'\|_{\infty}\leq 2$.
\end{proposition}
Let us now state a bound for the total variation distance between a random variable distributed according to (\ref{Gibbs}) and the standard normal law.
\begin{proposition}\label{goodlife}
Let $\mu$ be a probability law whose density is given by (\ref{Gibbs}). Let $X\sim\mu$.

\smallbreak Let $H\in\mathcal{C}^1(\mathbb{R}^d,\mathbb{R})$ be a function such that $H$ and its partial derivatives satisfies the condition (\ref{Cond1}) stated above (with $f=H$).

\smallbreak Let $ X_\infty$ be an independent copy of $X$, set $Y_0=\nabla\phi(X)$, $Y_{\infty}=\nabla\phi(X_{\infty})$  and $Y_t=e^{-t}\nabla\phi(X)+\sqrt{1-e^{-2t}}\nabla\phi(X_{\infty})$, $t>0$.

\smallbreak Let $\mathbb{E}_{\infty}$ denote the expectation with respect to $X_{\infty}$, and set $\mathbf{E}=\mathbb{E}\otimes\mathbb{E}_{\infty}$.

\smallbreak Let $F:=\frac{H(Y_0)-m}{\sigma}$ with $m=\mathbb{E}[H(Y_0)]$ and $\sigma^2=\Var[H(Y_0)]$, and let $N\sim \mathcal{N}(0,1)$ be a standard normal random variable. 

\smallbreak Then

\begin{equation}\label{boundTVGibbs}d_{TV}(F,N)\leq A+B \end{equation}
with
\begin{eqnarray*}
A&=& \frac{2}{\sigma^2}\sqrt{\Var\left(\int_0^{\infty}e^{-t}\langle ({\rm Hess}\phi)(X)\nabla H(Y_0),\mathbb{E}_{\infty}[\nabla H(Y_{t})]\rangle dt\right)},\\
B&=&\sup_{g}\big|\mathbf{E}[C(g)]\big|,
\end{eqnarray*}
with $$C(g)=\frac{1}{\sigma}\int_0^{\infty}e^{-2t}(g(F)-F){\rm Tr}\left[\big(({\rm Hess}\phi)(X)-({\rm Hess}\phi)(X_{\infty})\big)({\rm Hess}H)(Y_t)\right]dt$$
and where the supremum runs over the set of all functions $g\in\mathcal C^1$ such that $|g(x)|\leq \sqrt{\frac{\pi}{2}}$ and$|g'(x)|\leq 2$ for all $x\in\R$.
\end{proposition}
\begin{proof}
Without loss of generality, let us assume that $m=0$. Let $g\in \mathcal{C}^1(\mathbb{R})$ with $\|g\|_\infty\leq \sqrt{\frac\pi 2}$ and $\|g'\|_{\infty}\leq 2$. 
We can write

\begin{eqnarray*}
\mathbb{E}[Fg(F)]&=&\mathbf{E}\left[\frac{1}{\sigma}(H(Y_0)-H(Y_{\infty}))g\left(\frac{H( Y_0)}{\sigma}\right)\right]\\
&=&-\frac{1}{\sigma}\int_0^{\infty}\frac{d}{dt}\mathbf{E}\left[H( Y_t)g\left(\frac{H( Y_0)}{\sigma}\right)\right]dt\\
&=&\frac{1}{\sigma}\int_0^{\infty}e^{-t}\mathbf{E}\left[g\left(\frac{H(Y_0)}{\sigma}\right)\left\langle \nabla H(Y_t), Y_0\right\rangle\right]dt\\
&&-\frac{1}{\sigma}\int_0^{\infty}\frac{e^{-2t}}{\sqrt{1-e^{-2t}}}\mathbf{E}\left[g\left(\frac{H(Y_0)}{\sigma}\right)\left\langle \nabla H(Y_t), Y_{\infty}\right\rangle\right]dt
\end{eqnarray*}

Thanks to the Lemma \ref{IPP}, we have 
\begin{eqnarray*}
&&\mathbb{E}[Fg(F)]\\
&=&\frac{1}{\sigma}\int_0^{\infty}e^{-t}\mathbf{E}\left[{\rm Tr}\left(\nabla\left(x\rightarrow g\left(\frac{H(\nabla\phi(x))}{\sigma}\right)\right.\right.\right.\\
&&\hskip2cm\times\nabla H(e^{-t}\nabla\phi(x)+\sqrt{1-e^{-2t}}\nabla\phi(X_{\infty})\biggl)(X)\biggl)\biggl]dt\\
&&-\frac{1}{\sigma}\int_0^{\infty}\frac{e^{-2t}}{\sqrt{1-e^{-2t}}}\mathbf{E}\left[{\rm Tr}\left(\nabla\left(x\rightarrow g\left(\frac{H(Y_0)}{\sigma}\right)\right.\right.\right.\\
&&\hskip2cm\times\nabla H(e^{-t} \nabla\phi(X)+\sqrt{1-e^{-2t}}\nabla\phi(x))\biggl)( X_{\infty})\biggl)\biggl]dt\\
&=&\frac{1}{\sigma^2}\int_0^{\infty}e^{-t}\mathbf{E}\left[g'\left(\frac{H(Y_0)}{\sigma}\right)
\left\langle
({\rm Hess}\phi)(X)\nabla H(Y_0),\nabla H(Y_t)
\right\rangle\right]dt\\
&+&\frac{1}{\sigma}\int_0^{\infty}e^{-2t}\mathbf{E}\left[g(\frac{H(Y_0)}{\sigma}){\rm Tr}\left[\big({\rm Hess}\phi)(X)-({\rm Hess}\phi)(X_{\infty})\big)({\rm Hess}H)(Y_t)\right]\right]dt.
\end{eqnarray*}
We then have
\begin{eqnarray}
\label{911}\\
&&\mathbb{E}[Fg(F)-g'(F)]\notag\\
&=&\frac{1}{\sigma^2}\mathbf{E}\left[g'(F)\left(\int_0^{\infty}e^{-t}\left\langle ({\rm Hess}\phi)(X)\nabla H(Y_0),\nabla H(Y_t)\right\rangle dt-\sigma^2\right)\right]\notag\\
&+&\frac{1}{\sigma}\int_0^{\infty}e^{-2t}\mathbf{E}\left[g(F){\rm Tr}\left[\big(({\rm Hess}\phi)(X)-({\rm Hess}\phi)(X_{\infty})\big)({\rm Hess}H)(Y_t)\right]\right]dt.\notag
\end{eqnarray}
Taking $g(x)=x$ in (\ref{911}), we get
\begin{eqnarray*}
&&\sigma^2
=\int_0^{\infty}e^{-t}\mathbf{E}\left\langle ({\rm Hess}\phi)(X)\nabla H(Y_0),\nabla H(Y_t)\right\rangle dt\\
&+&\sigma\int_0^{\infty}e^{-2t}\mathbf{E}\left[F\,{\rm Tr}\left[\big(({\rm Hess}\phi)(X)-({\rm Hess}\phi)(X_{\infty})\big)({\rm Hess}H)(Y_t)\right]\right]dt.
\end{eqnarray*}
Using Cauchy-Schwarz and Minkowski's inequalities, as well as the bound \eqref{boundTotVar}, we finally have
\begin{eqnarray*}
d_{TV}(F,N)&\leq& A+\sup_{g}\big|\mathbf{E}[C(Tg)]\big|,
\end{eqnarray*}
with $(Tg)(x)=g(x)-x$, $$C(h)=\frac{1}{\sigma}\int_0^{\infty}e^{-2t}h(F){\rm Tr}\left[\big(({\rm Hess}\phi)(X)-({\rm Hess}\phi)(X_{\infty})\big)({\rm Hess}H)(Y_t)\right]dt$$
and where the supremum runs over the set of all functions $g\in\mathcal C^1$ such that $|g(x)|\leq \sqrt{\frac{\pi}{2}}$ and$|g'(x)|\leq 2$ for all $x\in\R$. This is the stated bound.
\end{proof}

\section{Proof of Theorem \ref{main}}\label{mainProof}
In the following, each time $K\subset \R^d$ denotes a convex body such
that ${\rm diam}K\leq 2cd^\alpha$ for some $\alpha,c>0$, we assume
without loss of generality (translation) that $0\in K$ and $K\subset c\,d^\alpha \mathbb{B}^d$.

To prove our main result, we will make use of the following three lemmas.

\begin{lemma}\label{LBoundDist}
Assume that the convex body $K$ of $\R^d$ is such that $K\subset c\,d^\alpha \mathbb{B}^d$ with $c>0$ (independent of $d$) and $\alpha\in [0,\frac12)$, 
and let $X_K:\Omega\to\R^d$ be a random vector distributed according to the Hadwiger-Wills density
(\ref{muK}). Let $\theta>0$ be such that $\ln(\frac{\sqrt{2\pi}}{\theta})<-1$. Then, as $d\to\infty$,
\begin{equation}\label{CondTheta}\mathbb{P}\left\{{\rm dist}(X_K,K)\leq\frac{\sqrt{d}}{\theta}\right\}=o\left(e^{-\frac{d}{2}}\right).\end{equation}
\end{lemma}
\begin{proof}
For all $x\in\mathbb{R}^d$, the function $\psi_x:K\mapsto e^{-\pi {\rm dist}^2(x,K)}$ is increasing for the inclusion 
(i.e., $C\subset K\implies \psi_x(C)\leq\psi_x(K)$).  
Moreover, for all $x\in\R^d$ we have $x=\Pi_K(x)+{\rm dist}(x,K)n(\Pi_K(x))$, where $n(u)$ denotes the outward pointing unit normal at $u\in \partial K$; in particular, $\|x\|\leq c\,d^{\alpha} + {\rm dist}(x,K)$ for all $x\in\R^d$.
Consequently, we have that
\begin{eqnarray*}
&&\mathbb{P}\left\{{\rm dist}(X_K,K)\leq\frac{\sqrt{d}}{\theta} \right\}\\
&\leq &\frac{1}{W(K)}\int_{\left(c\,d^\alpha+\frac{\sqrt{d}}{\theta}\right)\mathbb{B}^d}e^{-\pi {\rm dist}^2(x,K)}dx\\
&=&\frac{1}{W(K)}\left(\int_{c\,d^\alpha \mathbb{B}^d}e^{-\pi {\rm dist}^2(x,K)}dx+\int_{\left(c\,d^\alpha+\frac{\sqrt{d}}{\theta}\right)\mathbb{B}^d\setminus c\,d^\alpha \mathbb{B}^d}e^{-\pi {\rm dist}^2(x,K)}dx\right)\\
&\leq&\frac{1}{W(\{0\})}\left((c\,d^{\alpha})^d\kappa_d+\int_{\left(c\,d^\alpha+\frac{\sqrt{d}}{\theta}\right)\mathbb{B}^d\setminus c\,d^\alpha \mathbb{B}^d}e^{-\pi {\rm dist}^2(x,K)}dx\right)\\
&\leq&(c\,d^{\alpha})^d\kappa_d+\int_{\left(c\,d^\alpha+\frac{\sqrt{d}}{\theta}\right)\mathbb{B}^d\setminus c\,d^\alpha \mathbb{B}^d}e^{-\pi {\rm dist}^2(x,c\,d^\alpha \mathbb{B}^d)}dx\\
&=& (c\,d^{\alpha})^d \kappa_d+d\kappa_d\int_0^{\frac{\sqrt{d}}{\theta}}e^{-\pi r^2}(r+c\,d^\alpha)^{d-1}dr
\end{eqnarray*}
(where we first used that we are integrating a radial function, and then applied a linear change of variable). 
Combining $\Gamma(x+1)\sim \sqrt{2\pi}x^{x+\frac12}e^{-x}$ as $x\to\infty$ with the identity $\kappa_d=\frac{\pi^{\frac{d}2}}{\Gamma(\frac{d}2+1)}$ leads to
the equivalent $\kappa_d\sim \frac{1}{\sqrt{\pi d}}\left(\frac{2\pi e}{d}\right)^{\frac{d}2}$ as $d\to\infty$.
Using moreover the bound $$\int_0^{\frac{\sqrt{d}}{\theta}}e^{-\pi r^2}(r+c\,d^\alpha)^{d-1}dr\leq \left(\frac{\sqrt{d}}{\theta}+c\,d^{\alpha}\right)^{d-1}\int_0^\infty e^{-\pi r^2}dr = \left(\frac{\sqrt{d}}{\theta}+c\,d^{\alpha}\right)^{d-1}$$
 and the fact $\theta$ is such that $\ln(\frac{\sqrt{2\pi}}{\theta})<-1$, we arrive to the announced conclusion.
\end{proof}
\begin{definition}\label{surfaceMeasure}
Fix $r>0$ and let $K\subset \R^d$ be a convex body. Set $p(r)=(p_{0}(r),\ldots,p_{d-1}(r))$, where
$$p_i(r)=\frac{1}{P(r)}(d-i)v_i(K)\kappa_{d-i}r^i$$
with $P(r)$ such that $\sum_{i=0}^{d-1} p_i(r)=1$. By convention, we also define $p(\infty)$ as $(0,\ldots,0,1)$.

\end{definition}

\begin{lemma}\label{AnULClaw}
For all $r>0$, the probability measure $(p_i(r))_{0\leq i\leq d-1}$ defined above is ultra log-concave.
\end{lemma}
\begin{proof} We need to prove that for all $i\in\{0,\ldots,d-2\}$,
$$(v_i(K)(d-i)\kappa_{d-i})^2\geq \frac{i+1}{i}v_{i-1}(K)v_{i+1}(K)(d-i-1)(d-i+1)\kappa_{d-i-1}\kappa_{d-i+1}.$$
Since we already know that the sequence of intrinsic volumes is ultra log-concave (see, e.g., \cite{chevet1976processus}[Lemma 4.2]), it is enough to show that
$$(d-i)^2\geq(d-i-1)(d-i+1)$$
(which is trivial) and
$$\kappa^2_{d-i}\geq\kappa_{d-i-1}\kappa_{d-i+1}.$$
The last inequality above is equivalent to showing that 
$$\frac{\Gamma\left(1+\frac{d-i-1}{2}\right)}{\Gamma\left(1+\frac{d-i}{2}\right)}\geq\frac{\Gamma\left(1+\frac{d-i}{2}\right)}{\Gamma\left(1+\frac{d-i+1}{2}\right)}.$$
To prove this inequality,
we can observe that the expectation of the $\chi$-distribution of $j$ degrees of freedom is given by
$$\mathbb{E}[\|N_j\|]=\sqrt{2}\frac{\Gamma\left(\frac{j+1}{2}\right)}{\Gamma\left(\frac{j}{2}\right)}\quad (N_j\sim \mathcal{N}_{j}(0,I_j)).$$
Moreover, the function $\varphi_j:\R\to \R$ defined by $\varphi_j(u)=\mathbb{E}\big[\sqrt{\|N_j\|^2+u^2}\big]$ satisfies $\varphi_j'(u)\geq 0$ if and only if $u\geq 0$.
This implies that $\varphi_j(u)\geq \varphi_j(0)$ for all $u\in\R$, from which it comes that $\mathbb{E}[\|N_j\|]\leq \mathbb{E}[\|N_{j+1}\|]$.
As a result,
$$\frac{\Gamma\left(1+\frac{d-i-1}{2}\right)}{\Gamma\left(1+\frac{d-i}{2}\right)}=\frac{1}{\mathbb{E}[\|N_{d-i+1}\|]}\geq \frac{\Gamma\left(1+\frac{d-i}{2}\right)}{\Gamma\left(1+\frac{d-i+1}{2}\right)}=\frac{1}{\mathbb{E}[\|N_{d-i+2}\|]},$$
which concludes the proof.
\end{proof}

\begin{lemma}\label{crucialLemma}
Let $K\subset\R^d$ be a convex body, 
and let $X_K:\Omega\to\R^d$ be a random vector distributed according to the Hadwiger-Wills density
(\ref{muK}). Set $F=\frac{{\rm dist}^2(X_K,K)-\E[{\rm dist}^2(X_K,K)]}{\sqrt{\Var({\rm dist}^2(X_K,K))}}$. 
Let $X_{\infty}$ be an independent copy of $X_K$ and set
$$Y_t=e^{-t}(X_K-\Pi_K(X_K))+\sqrt{1-e^{-2t}}(X_{\infty}-\Pi_K(X_{\infty})).$$ 
Then
\begin{equation}\label{BoundTVDist}
d_{TV}(F,\mathcal{N}(0,1))\leq A_K+B_K,
\end{equation}
with 
\begin{eqnarray*}
A_K&=&\frac{4}{\sigma^2}\sqrt{\Var\left(\int_0^{\infty}e^{-t}\langle Y,\mathbb{E}_{\infty}[Y_t]\rangle dt\right)},\\
B_K&=&\frac{1.14}{\sigma}\sqrt{\Var\left(e_p\left(\frac{1}{{\rm dist}(X_K,K)}\right)\right)},
\end{eqnarray*}
where $Y=Y_0$ by convention and $e_p$ is the function defined by 
\begin{equation}\label{ep}
e_p(r)=\sum_{i=0}^{d-1}ip_i(r),
\end{equation}
with $p_i(r)$ defined in Definition \ref{surfaceMeasure}.
\end{lemma}
\begin{proof} The proof of this lemma is divided into two steps: first, we prove it for polytopes, and then we extend it for general convex sets through an approximation argument.

\medskip
\textit{\underline{Step 1: Case where $K$ is a polytope.}}
Let $K\subset\mathbb{R}^d$ be a polytope. Using the notations and results from Proposition \ref{polytopesProp},
 we have
 \begin{eqnarray*}
 &&\mathbb{P}[{\rm Tr}(\nabla\Pi_K(X_K))=i\,|\,{\rm dist}(X_K,K)]\\
 &=&\mathbb{P}[X_K\in C_{F}\,\mbox{ for some }F\in \mathcal F_i\,|\,{\rm dist}(X_K,K)]\\
&=&\frac{{\rm Vol}_{d-1}(\cup_{F\in\mathcal F_i}C_F\cap S_{{\rm dist}(X_K,K)})}{{\rm Vol}_{d-1}(S_{{\rm dist}(X_K,K)})}
=\frac{(d-i)v_i(K)\kappa_{d-i}\left(\frac{1}{{\rm dist}(X_K,K)}\right)^i}{\sum_{j=0}^{d-1}(d-j)v_j(K)\kappa_{d-j}\left(\frac{1}{{\rm dist}(X_K,K)}\right)^j}\\
&=&p_i\left(\frac{1}{{\rm dist}(X_K,K)}\right).\end{eqnarray*}
Then, on $\{{\rm dist}(X_K,K)>0\}$,
\begin{equation}\label{crucialRelationship}\mathbb{E}[{\rm Tr}(\nabla\Pi_K(X_K))\,|\,{\rm dist}(X_K,K)]=\sum_{i=0}^{d-1}ip_i\left(\frac{1}{{\rm dist}(X_K,K)}\right)=e_p\left(\frac{1}{{\rm dist}(X_K,K)}\right).
\end{equation}

We have that ${\rm dist}^2(X,K)=\|X-\Pi_K(X)\|^2=H(\nabla\phi(X))$ with 
$$H:x\mapsto \frac{1}{4}x^2\quad\mbox{and}\quad\phi:x\mapsto {\rm dist}^2(x,K).$$ 
We can apply Proposition \ref{goodlife} to obtain
\begin{eqnarray*}
d_{TV}(F,N)\leq A_K+B_K,
\end{eqnarray*}
with 
\begin{eqnarray*}
A_K&=&\frac{2}{\sigma^2}\sqrt{\Var\left(\int_0^{\infty}e^{-t}\langle Y,\mathbb{E}_{\infty}[Y_t]\rangle dt\right)},
\end{eqnarray*} 
(we also used (\ref{relation})) and
\begin{eqnarray*}
B_K&=&\frac{1}{2\sigma}\sup_g
\left|
{\bf E}\left\{ (g(F)-F){\rm Tr}\big[\nabla\Pi_K(X_K)-\nabla \Pi_K(X_\infty)]\right\}
\right|
\end{eqnarray*}
Moreover, 
\begin{eqnarray*}&&\mathbf{E}[\left(g(F)-F\right){\rm Tr}(\nabla\Pi_K(X_K)-\nabla\Pi_K(X_{\infty}))]\\
&=&\mathbf{E}\big[\left(g(F)-F\right)\\
&&\times\left(\mathbb{E}[{\rm Tr}(\nabla\Pi_K(X_K)|{\rm dist}(X_K,K)]-\mathbb{E}_{\infty}[\mathbb{E}_{\infty}[\nabla\Pi_K(X_{\infty}))|{\rm dist}(X_{\infty},K)]]\right)\big]\\
&\leq&\sqrt{\mathbb{E}[\left(g(F)-F\right)^2]}\sqrt{\Var(\mathbb{E}[{\rm Tr}(\nabla\Pi_K(X_K))|{\rm dist}(X_K,K)])}\\
&\leq&\sqrt{\mathbb{E}[\left(g(F)-F\right)^2]}\sqrt{\Var\left(e_p\left(\frac{1}{{\rm dist}(X_K,K)}\right)\right)}
\end{eqnarray*}
where we conditioned with respect to ${\rm dist}(X_K,K)$, then used the Cauchy-Schwarz inequality and then the identity \eqref{crucialRelationship}. Using the bounds on $g$, the conclusion follows.

\medskip
\textit{\underline{Step 2: General case}}.
Let $K\subset\mathbb{R}^d$ be a convex body. We know that any convex body can be approximated in the Hausdorff metric by polytopes (see e.g. \cite{schneider1981approximation}) that is, there is a sequence $(P_n)_{n\in\mathbb{N}^*}$ of polytopes in $\mathbb{R}^d$ such that for all $n$, $P_n\subset K$ and $$\max(\max_{x\in P_n}\min_{x\in K}\|x-y\|,\max_{x\in K}\min_{x\in P_n}\|x-y\|)\underset{n\rightarrow\infty}\longrightarrow 0.$$
As a consequence
\begin{itemize}
\item For all $i\in\{0,d\}$, the intrinsic volume $V_i(P_n)$ converges to $V_i(K)$ as $n\rightarrow\infty$,
\item $\sup_{x\in\mathbb{R}^d}\big|{\rm dist}(x,K)-{\rm dist}(x,P_n)\big|\underset{n\rightarrow\infty}\longrightarrow 0$.
\end{itemize}
For all $n\in\mathbb{N}^*$, let us denote by $X_n$ a random vector having the distance law with respect to $P_n$ and $X$ having the distance law with respect to $K$. Then, 
\begin{itemize}
\item the densities of $X_n$ converge pointwise to the density of $X$.
\item the densities of $X_n$ are uniformly bounded in $n$ by the function $x\mapsto e^{-\pi {\rm dist}^2(x,K)}$ which verifies $\int_{\mathbb{R}^d}\|x\|^ke^{-\pi {\rm dist}^2(x,K)}dx<\infty.$ 
\end{itemize}
As a consequence, by dominated convergence, we have that 
$$\mathbb{E}[{\rm dist}^2(X^n,P_n)]\underset{n\rightarrow\infty}\longrightarrow\mathbb{E}[{\rm dist}^2(X,K)],$$
$$\Var({\rm dist}^2(X^n,P_n))\underset{n\rightarrow\infty}\longrightarrow\Var({\rm dist}^2(X,K)),$$
$$d_{TV}(F,N)\leq\lim\inf_{n\in\mathbb{N}^*} d_{TV}(F_n,N)\leq\lim\inf_{n\in\mathbb{N}^*}(A_{P_n}+B_{P_n})=A_K+B_K.$$
where the third inequality in the last line comes from Step 1. This conclude the proof.
\end{proof}

\begin{proof}[\nopunct]\textit{Proof of Theorem \ref{main}:}
We have that $d_{TV}(F,N)\leq A_{K}+B_{K}$ with $A_{K}$ and $B_{K}$ given in Lemma \ref{crucialLemma}. We will bound separately the terms $A_K$ and $B_K$ to obtain our main result. For simplicity, we write $A$ and $B$ instead of $A_K$ and $B_K$.

\medskip
\textit{\underline{Step 1: Bounding the term $A$}}.
To bound this standard deviation, we shall apply the Brascamp-Lieb inequality \eqref{BC}. However, since $X$ is not a strongly log-concave vector, we need to introduce a modified version. Fix $d\in\mathbb{N}^*$, let $\epsilon>0$ and let $X_{\epsilon}$ be a random vector with density proportional to $e^{-\phi_{\epsilon}}$, with $\phi_{\epsilon}:x\mapsto \pi {\rm dist}^2(x,K)+\epsilon\pi \|x\|^2$. We have $$({\rm Hess}~\phi_{\epsilon})(x)=2\pi\big[(1+\epsilon)I-\nabla\Pi_{K}(x)\big],$$
with $I$ the $d\times d$-identity matrix. 
Since $\Pi_{K}$ is a contracting operator, we have that $X_{\epsilon}$ is $2\pi\epsilon$-strongly log-concave.
Furthermore, thanks to the dominated convergence theorem, we can write
$$A=\frac{4}{\sigma^2}\lim_{\epsilon\rightarrow 0}\sqrt{\Var\left(\int_0^{\infty}e^{-t}\langle Y^\epsilon,\mathbb{E}_{\infty}[Y^\epsilon_t]\rangle dt\right)}.$$
Set 
\begin{eqnarray*}
f_\epsilon(x)&=&\int_0^\infty \langle e^{-t} \nabla\phi_\epsilon,e^{-t}\nabla\phi_\epsilon(x)+\sqrt{1-e^{-2t}}\mathbb{E}(\nabla \phi_\epsilon(X^\epsilon))\rangle dt\\
&=&\|\nabla \phi_\epsilon(x)\|^2 + \langle \nabla\phi_\epsilon(x), \mathbb{E}\nabla \phi_\epsilon(X^\epsilon)\rangle.
\end{eqnarray*}
We can now apply the Brascamp-Lieb inequality \eqref{BC} to obtain:
\begin{eqnarray*}
&&\Var\left(\int_0^{\infty}e^{-t}\langle Y^\epsilon,\mathbb{E}_{\infty}[Y^\epsilon_t]\rangle dt\right)=\Var(f_\epsilon(X^\epsilon))\\
&\leq& \mathbb{E}\left[\nabla f_\epsilon(X^{\epsilon})^T\left({\rm Hess}\,\phi_\epsilon(X^\epsilon)\right)^{-1}\nabla f_\epsilon(X^{\epsilon})\right].
\end{eqnarray*}
We have
\begin{eqnarray*}
\nabla f_\epsilon(x)&=& 2({\rm Hess}\,\phi_\epsilon(x))\nabla \phi_\epsilon(x) + ({\rm Hess}\,\phi_\epsilon(x))\mathbb{E}(\nabla \phi_\epsilon(X^\epsilon))\\
\nabla\phi_\epsilon(x)&=&2\pi\big((1+\epsilon)x- \Pi_K(x)\big)\\
{\rm Hess}\,\phi_\epsilon(x)&=&2\pi\big((1+\epsilon)I - \nabla \Pi_K(x)\big).
\end{eqnarray*}
In particular,
$$
\big({\rm Hess}\,\phi_\epsilon(x)\big)^{-1}\nabla f_\epsilon(x) = 2 \nabla \phi_\epsilon(x) + \mathbb{E}\nabla \phi_\epsilon(X^\epsilon).
$$
As a result,
\begin{eqnarray*}
\Var(f_\epsilon(X^\epsilon)) &\leq&\sqrt{10} \sqrt{\mathbb{E}\|\nabla f_\epsilon(X^\epsilon)\|^2} \sqrt{\mathbb{E}\|\nabla \phi_\epsilon(X^\epsilon)\|^2}.
\end{eqnarray*}
But, for $\epsilon<1$ (say)
\begin{eqnarray*}
\|\nabla f_\epsilon(x)\| &\leq& 2\times 2\pi\|\nabla \phi_\epsilon(x)\|+ 2\pi \mathbb{E}\|\nabla\phi_\epsilon(X^\epsilon)\|\\
\Rightarrow \mathbb{E}\|\nabla f_\epsilon(X^\epsilon)\|^2 &\leq&10(2\pi)^2 \,\mathbb{E}\|\nabla\phi_\epsilon(X^\epsilon)\|^2,
\end{eqnarray*}
and so
$$
\Var(f_\epsilon(X^\epsilon)) \leq 20\pi\,\E\|\nabla \phi_\epsilon(X^\epsilon)\|^2.
$$
Finally,
\begin{eqnarray*}
A&\leq& \frac4{\sigma^2}\sqrt{20\pi}\sqrt{\E\|\nabla \phi(X_K)\|^2}\\
&\leq& \frac4\sigma^2\sqrt{20\pi}\times 2\pi\sqrt{\mathbb{E}\big[{\rm dist}^2(X_K,K)\big]}\leq \frac{16\sqrt{5}\pi^{3/2}}{\sigma}.
\end{eqnarray*}
where the last inequality follows from Lemma \ref{relationsVarEsp}(2). Moreover, since $K_d\subset d^\alpha\mathbb{B}$, we have, thanks to Lemma \ref{relationsVarEsp} again and  Corollary \ref{BoundIntr}, 
\begin{equation}\label{LBoundSigma1}\sigma^2\geq\mathbb{E}[{\rm dist}^2(X_K,K)]=\frac{d-\mathbb{E}[V_K]}{2}\geq \frac{d-v_1(d^\alpha\mathbb{B})}{2}\geq \frac{d-\frac{d^{1+\alpha}}{\Gamma(1+\frac{d}{2})}}{2}.\end{equation}
Using Stirling formula, we have that there is some $d_0\in\mathbb{N}^*$ such that for all $d\geq d_0$,
\begin{equation}\label{LBoundSigma2}\sigma\geq Kd^\frac{1}{2}\end{equation} for some $K>0$, and then $A=O\left(\frac{1}{\sqrt{d}}\right)$ as $d\rightarrow\infty$.

\medskip 
\textit{\underline{Step 2: Bounding the term $B$:}}
We now want to prove that 
\begin{equation}\label{toBfree}
\frac{1}{\sigma}\sqrt{\Var\left(e_p\left(\frac{1}{{\rm dist}(X_K,K)}\right)\right)}
= O\left(\frac{1}{d^{\frac{1}{2}-\alpha}}\right).
\end{equation}
Taking \eqref{LBoundSigma2} into account, it is enough to show that $$\Var\left(e_p\left(\frac{1}{{\rm dist}(X_K,K)}\right)\right)=O(d^{2\alpha}).$$

For this, let us first smoothen the random variable $\frac{1}{{\rm dist}(X_K,K)}$ in the following way: let us define $\nu:\mathbb{R}_+\to\mathbb{R}$ as the following function:
\begin{equation*}\nu(x):=\left\{
      \begin{array}{cl}
        x &\mbox{ if }x\leq 1\\
        1 &\mbox{ else}.
      \end{array}
    \right.
    \end{equation*}
Observe from (\ref{ep}) that $e_p\left(\frac{1}{{\rm dist}(x,K)}\right)\leq d$ for all $x$. Then, 
\begin{eqnarray*}&&\left|\Var\left(e_p\left(\frac{1}{{\rm dist}(X_K,K)}\right)\right)-\Var\left(e_p\left(\nu\left(\frac{1}{{\rm dist}(X_K,K)}\right)\right)\right)\right|\\&\leq&\left|\mathbb{E}\left[e_p\left(\frac{1}{{\rm dist}(X_K,K)}\right)^2-e_p\left(\nu\left(\frac{1}{{\rm dist}(X_K,K)}\right)\right)^2\right]\right|\\
&&+\left|\mathbb{E}\left[e_p\left(\frac{1}{{\rm dist}(X_K,K)}\right)\right]-\mathbb{E}\left[e_p\left(\nu\left(\frac{1}{{\rm dist}(X_K,K)}\right)\right)\right]\right|\\&&\times\left|\mathbb{E}\left[e_p\left(\frac{1}{{\rm dist}(X_K,K)}\right)\right]+\mathbb{E}\left[e_p\left(\nu\left(\frac{1}{{\rm dist}(X_K,K)}\right)\right)\right]\right|\\
\end{eqnarray*}
\begin{eqnarray*}
&\leq&\left|\mathbb{E}\left[\left(e_p\left(\frac{1}{{\rm dist}(X_K,K)}\right)^2-e_p\left(\nu\left(\frac{1}{{\rm dist}(X_K,K)}\right)\right)^2\right)
{\bf 1}_{\{{\rm dist}(X_K,K)<1\}}
\right]\right|\\
&&+\left|\mathbb{E}\left[\left(e_p\left(\frac{1}{{\rm dist}(X_K,K)}\right)-e_p\left(\nu\left(\frac{1}{{\rm dist}(X_K,K)}\right)\right)\right)
{\bf 1}_{\{{\rm dist}(X_K,K)<1\}}
\right]\right|\\&&\times(2d)\\
&\leq& 3d^2\,
\mathbb{P}({\rm dist}(X_K,K)<1) \leq 3d^2\,
\mathbb{P}\left({\rm dist}(X_K,K)<\frac{\sqrt{d}}{\theta}\right) \leq 3d^2\,e^{-\frac{d}2},
\end{eqnarray*} (The last inequality comes from Lemma \ref{LBoundDist}). 

We can now repeat the same procedure as in Step 1 (that is, to apply the Brascamp-Lieb inequality to a modified version $X^{\epsilon}$ and then use the dominated convergence theorem) to obtain:
\begin{eqnarray}\label{estimate}
&&\Var\left(e_p\left(\nu\left(\frac{1}{{\rm dist}(X_K,K)}\right)\right)\right)\\
&\leq&\frac1{2\pi} \lim_{\epsilon\rightarrow 0}\mathbb{E}[\nabla f_{\nu}(X^{\epsilon})^T((1+\epsilon)I-\nabla\Pi_K(X^{\epsilon}))^{-1}\nabla f_{\nu}(X^{\epsilon})]\notag\end{eqnarray}
with $f_{\nu}(y)=e_p\left(\nu\left(\frac{1}{{\rm dist}(y,K)}\right)\right)$. 
We have for all $y\in\mathbb{R}^d,$
\begin{eqnarray*}
&&\nabla f_\nu(y)\\
&=&\frac{(\Pi_K(y)-y)\mathbf{1}_{\{\|y-\Pi_K(y)\|\geq 1\}}}{{\rm dist}^3(y,K)}\left[\frac{\sum_{i=0}^{d-1}i^2v_i(K)(d-i)\kappa_{d-i}\left(\frac{1}{{\rm dist}(y,K)}\right)^{i-1}}{P\left(\frac{1}{{\rm dist}(y,K)}\right)}\right.\\
&&\left.\hskip2cm\quad-\frac{{\rm dist}(y,K)\left(\sum_{i=0}^{d-1}iv_i(K)(d-i)\kappa_{d-i}\left(\frac{1}{{\rm dist}(y,K)}\right)^{i}\right)^2}{P\left(\frac{1}{{\rm dist}(y,K)}\right)^2}\right]\\
&&=\frac{(\Pi_K(y)-y)\mathbf{1}_{\{\|y-\Pi_K(y)\|\geq 1\}}}{{\rm dist}(y,K)}var_p\left(\frac{1}{{\rm dist}^2(y,K)}\right),
\end{eqnarray*}
with $var_p(r)=\sum_{i=0}^{d-1}i^2p_i(r)-e_p(r)^2$.
Pluging this into \eqref{estimate} yields:
\begin{eqnarray*}
&&\Var\left(e_p\left(\nu\left(\frac{1}{{\rm dist}(X_K,K)}\right)\right)\right)\\
&\leq&\frac1{2\pi}\,\mathbb{E}\left[
\nabla f_\nu(X_K)^T(I-\nabla\Pi_K(X_K))^{-1}\nabla f_\nu(X_K)
\right]\\
&=&\frac{1}{2\pi}\mathbb{E}\left[\|\nabla f_\nu(X_K)\|^2\right]\quad\mbox{(using (\ref{relation}))}\\
&=& \frac1{2\pi}\mathbb{E}\left[{\frac{\mathbf{1}_{\{{\rm dist}(X_K,K)\geq 1\}}}{{\rm dist}^2(X_K,K)}var^2_p\left(\frac{1}{{\rm dist}(X_K,K)}\right)}\right].
\end{eqnarray*}
Invoking again Lemma \ref{LBoundDist}, we have
\begin{eqnarray}
&&\Var\left(e_p\left(\nu\left(\frac{1}{{\rm dist}(X_K,K)}\right)\right)\right)\notag\\
&\leq&\frac{1}{2\pi}\left\{d^4e^{-\frac{d}2}+
\mathbb{E}\left[{\frac{\mathbf{1}_{\{{\rm dist}(X_K,K)>\frac{\sqrt{d}}{\theta}\}}}{{\rm dist}^2(X_K,K)}var^2_p\left(\frac{1}{{\rm dist}(X_K,K)}\right)}\right]\right\}.\label{estimate2}
\end{eqnarray}
We have that $p(x)$ is an ultra-log concave measure for all $x>0$ thanks to Lemma \ref{AnULClaw}. 
Morever
$K\subset d^{\alpha}\mathbb{B}^d$. By monotonicity of the intrinsic volumes and Proposition \ref{ULCVar}, we deduce for all $r>0$ the bound
$$var_p(r)\leq \frac{p_1(r)}{p_0(r)}\leq r\,d^{\alpha}\,v_1(\mathbb{B}^d)\,\frac{\kappa_{d-1}}{\kappa_d}\,\frac{d-1}{d}=O(d^{1+\alpha})r.$$
Plugging this into \eqref{estimate2} yields
$$\Var\left(e_p\left(\nu\left(\frac{1}{{\rm dist}(X_K,K)}\right)\right)\right)=O(d^{2\alpha}),$$
implying in turn (\ref{toBfree}).

\bigskip

\textit{\underline{Step 3: Conclusion}}. By combining the bounds for $A$ and $B$ obtained in Steps 1 and 2, we obtain the rate of Theorem \ref{main}, that is,
$$d_{TV}(F,N)=O(d^{\alpha-\frac{1}{2}}).$$
\end{proof}

\bigskip

\noindent
{\bf Acknowledgments}. We thank Michele Stecconi for his help in proving the first point of Proposition \ref{polytopesProp}.
We were supported (in part) by the Luxembourg National Research Fund O18/12582675/APOGee.

\bibliographystyle{plain}

\end{document}